\documentclass{amsart}

\usepackage{a4wide}
\usepackage{amsmath}
\usepackage{amssymb}
\usepackage{amsthm}
\usepackage[initials]{amsrefs}
\usepackage{url}
\usepackage{color}

\newtheorem{rem}{Remark}

\newtheorem{lem}{Lemma}

\newtheorem{theo}{Theorem}

\newtheorem{innercustomlem}{Lemma}
\newenvironment{customlem}[1]
  {\renewcommand\theinnercustomlem{#1}\innercustomlem}
  {\endinnercustomlem}

\providecommand{\house}[1]{\left.\overline{\vbox{\vglue-1.4pt\hbox{$\!\left|\,#1\,\right|\!$}}}\right.}

\DeclareMathOperator{\ord}{ord}
\DeclareMathOperator{\PP}{P}

\title{Romanov's Theorem in Number Fields}

\author[M. G. Madritsch]{Manfred G. Madritsch}
\address[M. G. Madritsch]{
\noindent 1. Universit\'e de Lorraine, Institut Elie Cartan de Lorraine, UMR 7502, Vandoeuvre-l\`es-Nancy, F-54506, France;\newline
\noindent 2. CNRS, Institut Elie Cartan de Lorraine, UMR 7502, Vandoeuvre-l\`es-Nancy, F-54506, France}
\email{manfred.madritsch@univ-lorraine.fr}

\author[S. Planitzer]{Stefan Planitzer}
\address[S. Planitzer]{
Graz University of Technology, Institute of Mathematics A, Steyrergasse 30, 8010 Graz, Austria}
\email{planitzer@math.tugraz.at}



\date{\today}

\begin{document}

\begin{abstract}
Romanov proved that a positive proportion of the integers have a representation as a sum of a prime and a power of an arbitrary fixed positive integer. Rieger proved the analogous result for number fields. We will determine an explicit lower bound for the proportion of algebraic integers in a given number field, which are sums of a power of a fixed non-unit and a prime. Furthermore, we give an improved lower bound for the lower density of Gaussian integers that have a representation as a sum of a Gaussian prime and a power of $1+i$. Finally, similar to Erd\H{o}s, we construct an explicit arithmetic progression of Gaussian integers with odd norm such that almost all elements of this progression do not have a representation as the sum of a prime and a power of $1+i$.
\end{abstract}

\maketitle

\section{Introduction and statement of results}

The problem of determining the proportion of positive integers which
are of the form $p+g^k$ has quite a long history. Especially the case
of sums of primes and powers of two received prominent treatment.
Mathematicians dating back at least as far as
Euler~\cite{LetterTo} worked on this problem. In 1934
Romanov~\cite{UeberEinige} proved that the proportion of integers
$n \in \mathbb{N}$ of the form $n=p+g^k$ with $p \in \mathbb{P}$,
$k \in \mathbb{N}_0$, $1 < g \in \mathbb{N}$ is positive. For the case
$g=2$ Erd\H{o}s~\cite{OnIntegersRelated} and van der
Corput~\cite{OverHet} independently proved a counterpart to Romanov's
theorem stating that a positive proportion of the odd positive
integers cannot be represented as a sum of a prime and a power of
two.

Recently explicit lower bounds for the lower density of 
integers of the form $p+2^k$, i.e.
$$\liminf_{x\rightarrow \infty} \frac{\#\{n \leq x: n=p+2^k,p \in \mathbb{P},k \in \mathbb{N}_0\}}{x},$$
were published (see \cite{OnRomanoffs}, \cite{OnRomanovs},
\cite{OnIntegers}, \cite{AnEffective} and \cite{ANote}). Rieger~\cite{VerallgemeinerungZweier} proved a number field analogue of 
Romanov's result and quite recently Shparlinski and 
Weingartner~\cite{AnExplicit} proved an analogue for Romanov's theorem 
for polynomials over finite fields. Our aim is to determine 
an explicit lower bound for the proportion of algebraic integers in a 
given number field which are sums of a power of a fixed non-unit and 
a prime. Furthermore, in the case of Gaussian integers which are not the sum of a prime and a power of $1+i$, we also prove the Erd\H{o}s and van der Corput type result. Before we state the corresponding theorems we need to fix some standard notation.

For the rest of this paper $\mathbb{N}$ and $\mathbb{P}$ will have
their usual meaning denoting the set of positive integers and positive
primes respectively and $K$ will always denote a number field of
degree $n=r_1+2r_2$, where the non-negative integers $r_1$ and $r_2$,
as usual, denote the number of real and pairs of complex conjugate
embeddings of $K$. By $\mathcal{O}_K$ we denote the ring of integers
of $K$, $\mathcal{P}_K$ is the set of prime elements in
$\mathcal{O}_K$ and $\xi \in \mathcal{O}_K$ a fixed non-unit. The
letter $\mathfrak{p}$ with or without index in any case denotes a
prime ideal of $\mathcal{O}_K$. If $\mathfrak{a}$ is an ideal of
$\mathcal{O}_K$ we write $\mathfrak{a} \unlhd \mathcal{O}_K$ for short
and for $\zeta \in \mathcal{O}_K$, $(\zeta)$ denotes the ideal of
$\mathcal{O}_K$ which is generated by $\zeta$. For $\zeta \in K$,
$\mathcal{N}(\zeta)$ denotes the field norm of $\zeta$ and
$\sigma_i:K \rightarrow \mathbb{C}$, $1 \leq i \leq n$, are the
embeddings of $K$. We write
$\house{\zeta}:=\max_{1 \leq i \leq n}{|\sigma_i(\zeta)|}$ for the
house of $\zeta$. By $\omega$ we denote the number of roots of unity
in a given ring of integers, $R$ denotes its regulator, $h_K$ the
class number and $\Delta_K$ the discriminant. Furthermore $\rho_K$
denotes the residue of the Dedekind zeta function of $K$ at $s=1$. For
a subset $S\subset \mathcal{O}_K$ we define its lower density in
$\mathcal{O}_K$ by
$$\underline{d}(S):=\liminf_{x \rightarrow \infty}\frac{\#\{\zeta \in S:\house{\zeta} \leq x\}}{\#\{\zeta \in \mathcal{O}_K: \house{\zeta} \leq x\}}.$$
The upper density of $S$ in $\mathcal{O}_K$ is accordingly defined as
$$\overline{d}(s):=\limsup_{x \rightarrow \infty}\frac{\#\{\zeta \in S:\house{\zeta} \leq x\}}{\#\{\zeta \in \mathcal{O}_K: \house{\zeta} \leq x\}},$$ 
and we say that $S$ has density $d$ in $\mathcal{O}_K$ if $d=\underline{d}(S)=\overline{d}(S)$.

Our main result will be the following theorem in which we prove an
explicit lower bound on the lower density of algebraic integers 
in a given ring of integers $\mathcal{O}_K$ which are sums of a prime
and a power of a fixed non-unit. The existence of such a constant was
already proven by Rieger~\cite{VerallgemeinerungZweier}*{Satz 2}.

\begin{theo} \label{MainTheorem}
  Let $K$ be a number field, $\mathcal{O}_K$ its ring of integers,
  $\mathcal{P}_K$ the set of prime elements in $\mathcal{O}_K$ and
  $\xi \in \mathcal{O}_K$ neither zero nor a unit.  Then the lower density of
  algebraic integers $\zeta \in \mathcal{O}_K$ of the form
  $\zeta=\pi+\xi^k$ for $\pi \in \mathcal{P}_K$ is bounded from
  below by 
$$\frac{\left(\frac{\omega}{2nh_K2^{r_1}R\log \house{\xi}}\right)^2}{\frac{\pi^{r_2}2^{r_1+r_2}}{\sqrt{|\Delta_K|}}\left(\frac{\omega}{2nh_K2^{r_1}R\log \house{\xi}}+\frac{2^{r_1+r_2+2}\pi^{r_2+2n}}{\left(\frac{1}{4}-\epsilon\right)^2\sqrt{|\Delta_K|}6^n\rho_k^2}\prod_{\xi \in \mathfrak{p}}\left(1+\frac{1}{\mathcal{N}(\mathfrak{p})}\right)\frac{1}{8\log^2\house{\xi}}2^ne^{n\gamma}n!\right)},$$
where $\epsilon$ is an arbitrarily small positive real number. 
\end{theo} 

Concerning sums of primes and powers of two we want to treat the
corresponding case in the ring $\mathbb{Z}[i]$ in more detail. We use computational methods similar to those applied in~\cite{OnRomanoffs} to improve the bound in Theorem~\ref{MainTheorem} for Gaussian integers $\zeta \in \mathbb{Z}[i]$
with a representation of the form
$$\zeta = \pi + (1+i)^k$$
for $\pi \in \mathcal{P}_{\mathbb{Q}(i)}$ and $k \in \mathbb{N}_0$. Note
that similar to $2$ in $\mathbb{Z}$ the associates of $1+i$ are
exceptional primes in $\mathbb{Z}[i]$ in the sense that they are the only primes whose real- and imaginary parts have the
same parity. Furthermore we get the same trivial upper bound for the upper density of algebraic integers of the form $\pi +(1+i)^k$ that we get for integers of the form $p+2^k$. For $k \geq 2$ the real- and imaginary part of
$(1+i)^k$ are even and the real- and imaginary parts of primes in
$\mathbb{Z}[i]$, with the exception of the associates of $1+i$, have
different parity hence the norm of sums of those prime elements and powers
of $1+i$ is odd. Thus the upper density of algebraic integers
$\zeta \in \mathbb{Z}[i]$ with a representation of the form
$\zeta=\pi+(1+i)^k$ is at most $\frac12$. Before we state a theorem
concerning the proportion of Gaussian integers of the form
$\pi + (1+i)^k$ we introduce the functions $r_x(\zeta,\xi)$ and
$\eta_x(\zeta,\xi)$ counting the number of representations and being the
indicator function for integers with at least one representation, respectively, as follows
$$r_x(\zeta,\xi)=\#\left\{(\pi,k)\colon\begin{gathered}\pi \in \mathcal{P}_K, k \in \mathbb{N}_0, \zeta=\pi+\xi^k, \\ \house{\pi} \leq x, k \leq \frac{\log \frac{\sqrt{x}}{2}}{\log \house{\xi}}\end{gathered}\right\},$$
\begin{align*}
\eta_x(\zeta,\xi):=\begin{cases}
1, &\text{ if } r_x(\zeta,\xi)>0, \\
0, &\text{ otherwise}.
\end{cases}
\end{align*}
The condition
$k \leq \frac{\log \frac{\sqrt{x}}{2}}{\log \house{\xi}}$ implies that
$\house{\xi^k} \leq \frac{\sqrt{x}}{2}$ and is needed for technical
reasons. For Gaussian integers of the form $\pi+(1+i)^k$ we will improve the general bound 
from Theorem \ref{MainTheorem} to the following lower
bound.

\begin{theo} \label{ZIExplicitBoundTheorem}
In the case of $K=\mathbb{Q}(i)$ and $\xi=1+i$ we have
$$\liminf_{x \rightarrow \infty} \frac{\sum_{\house{\zeta} \leq x}\eta_x(\zeta,1+i)}{\#\{\zeta \in \mathbb{Z}[i]: \house{\zeta} \leq x\}} \geq 0.0029.$$
\end{theo} 
A direct application of Theorem~\ref{MainTheorem} would lead to the bound $0.000075$ instead.

Finally we prove the following analogue to the result of
Erd\H{o}s~\cite{OnIntegersRelated} and van der Corput~\cite{OverHet}.

\begin{theo} \label{ErdosVanDerCorputAnalogueTheorem}
A positive proportion of the algebraic integers $\zeta \in \mathbb{Z}[i]$ with odd norm $\mathcal{N}(\zeta)$ are not of the form $\pi+(1+i)^k$ for $k \in \mathbb{N}_0$, $\pi \in \mathcal{P}_{\mathbb{Q}(i)}$.
\end{theo} 

Following the ideas of Romanov~\cite{UeberEinige} we will make use of the Cauchy-Schwarz inequality and a sieve. Applying the Cauchy-Schwarz inequality yields
\begin{equation} \label{CSInequality}
\sum_{\house{\zeta} \leq x}{\eta_x(\zeta, \xi)} \geq \frac{\left(\sum_{\house{\zeta} \leq x}{r_x(\zeta,\xi)}\right)^2}{\sum_{\house{\zeta}\leq x}{r_x(\zeta,\xi)^2}},
\end{equation}
and we need to look for a lower bound for the numerator and an upper bound for the denominator of the right hand side of this inequality.

\section{A lower bound for the constant in Romanov's theorem for number fields}

We start with an upper bound for
$\sum_{\house{\zeta}\leq x}{r_x(\zeta, \xi)^2}$ by using results from
sieve theory. Since $r_x(\zeta, \xi)^2$ counts the number of pairs of
representations of the algebraic integer $\zeta$ as the sum of a prime
and a power of $\xi$ we have
$$\sum_{\house{\zeta}\leq x}{r_x(\zeta,\xi)^2} = |A|+|B|$$
where $A$ corresponds to pairs of equal representations and $B$ to
different ones, i.e.
\begin{equation} \label{SetA} A=\left\{(\pi,k)\colon\begin{gathered}\pi \in \mathcal{P}_K, k \in \mathbb{N}_0, \zeta=\pi+\xi^k, \\
      \house{\zeta}\leq x, \house{\pi} \leq x, k \leq \frac{\log
        \frac{\sqrt{x}}{2}}{\log \house{\xi}}\end{gathered}\right\}
\end{equation}
and
\begin{align} \label{SetB}
  B=\left\{(\pi_1,\pi_2,k_1,k_2)\colon\begin{gathered}\pi_i \in \mathcal{P}_K, k_i \in \mathbb{N}_0, \zeta=\pi_i+\xi^{k_i},\pi_1 \neq \pi_2,\\
      \house{\zeta} \leq x, \house{\pi_i}\leq x, k_i \leq \frac{\log
        \frac{\sqrt{x}}{2}}{\log \house{\xi}}\end{gathered}\right\} .
\end{align}
Finding an upper bound for the size of the set $B$ for fixed $k_1$ and
$k_2$ amounts to finding an upper bound for the number of distinct
primes $\pi_1$ and $\pi_2$ such that
$$\pi_1-\pi_2=\xi^{k_2}-\xi^{k_1}.$$
Note that with our restriction on the exponents $k_i$ and with the
triangle inequality for the house function we have that
$\house{\xi^{k_1}-\xi^{k_2}} \leq \sqrt{x}$. An upper bound for the
number of solutions in this case is given by the following Theorem:
\begin{theo} \label{NFSieve} Let $K$ be a number field of degree $n$,
  $\mathcal{O}_K$ its ring of integers, $0<x \in \mathbb{R}$ and
  $\zeta \in \mathcal{O}_K$ such that $\house{\zeta} \leq
  \sqrt{x}$.
  We denote by $\PP(\zeta,x)$ the number of solutions of the equation
  \[\zeta = \pi_1-\pi_2\]
  where $\pi_i \in \mathcal{P}_K$, $\house{\pi_i} \leq x$. Then there
  exists a constant $\kappa$ depending only on $K$ such that
  \[\PP(\zeta,x)\leq \kappa \frac{x^n}{\log^2x}\prod_{\mathfrak{p}| (\zeta)}\left(1+\frac{1}{\mathcal{N}(\mathfrak{p})}\right)(1+o_K(1)).\]
  In the general case we may choose $\kappa=\frac{2^{r_1+r_2+2}\pi^{r_2+2n}}{\left(\frac{1}{4}-\epsilon\right)^2\sqrt{|\Delta_K|}6^n\rho_K^2}$ and if $K=\mathbb{Q}(i)$ 
  the choice $\kappa=368.333$ is admissible.
\end{theo}
\begin{rem}
  The first part of the proof of Theorem \ref{NFSieve} is in large
  parts the same as Tatuzawa's proof of~\cite{AdditivePrime}*{Theorem
  1}. However, in his proof Tatuzawa makes use of the
  restriction $\house{\zeta} \leq c_K\sqrt[n]{\mathcal{N}(\zeta)}$ for
  a constant $c_K$ depending only on $K$ which in our application of
  this result will not be satisfied in general. Besides of the result
  of Tatuzawa many of the details appearing in the first part of the
  proof below can be found in the proof of Lemma 3.2 in Wang's
  book~\cite{DiophantineEquations}.

  The second part of the proof of Theorem \ref{NFSieve} works exactly
  as the proof of~\cite{VerallgemeinerungDer3}*{Satz 16} by applying
  Selberg's Sieve method for number fields (for a detailed description
  of the method see \cites{VerallgemeinerungDer1,
    VerallgemeinerungDer2, VerallgemeinerungDer3}) instead of the
  approach of Rademacher~\cite{UeberDie} using Brun's sieve that was
  applied by Tatuzawa in his proof.

  We therefore do not give all the details in the proof below, they
  can be found in the corresponding works of Tatuzawa, Rademacher,
  Wang and Rieger. The reason why we sketch the proof nontheless is
  because we need to work out an explicit sieve constant for later
  use. Furthermore our restrictions are slightly different from those
  in Tatuzawa's Theorem~\cite{AdditivePrime}*{Theorem 1} as well as
  Rieger's Theorems~\cite{VerallgemeinerungDer3}*{Satz 16 and Satz 17}.
\end{rem}

\begin{proof}[Proof of Theorem \ref{NFSieve}]

The proof will work in two steps. First we have to count the number of
possible pairs $(\zeta,\pi_2)$. In the second step we will use the
Selberg sieve in order to get the desired upper bound. 

\textbf{1. Counting Lattice Points.}  For some ideal
$\mathfrak{a} \unlhd \mathcal{O}_K$ and some element
$\beta \in \mathcal{O}_K$ we start by finding the asymptotic number
of algebraic integers $\xi \in \mathcal{O}_K$ such that
\begin{equation} \label{LatticePointEquations}
\xi \equiv \beta \bmod \mathfrak{a}, \quad \house{\xi} \leq x, \quad \house{\xi + \zeta} \leq x.
\end{equation}
We use the fact that there exists a constant $c_K$ depending only on
$K$ such that any ideal $\mathfrak{a} \unlhd \mathcal{O}_K$ has an
integral basis $\alpha_1, \ldots, \alpha_n$ such that
$ \house{\alpha_j} \leq c_K\sqrt[n]{\mathcal{N}(\mathfrak{a})}$ for all
$1 \leq j \leq n$ (for a proof of this see the first part of the proof
of~\cite{DiophantineEquations}*{Lemma 3.2}). Using such a basis
$(\alpha_j)_{j=1}^n$ and with
\begin{align*}
  &u_j=x_1\alpha_1^{(j)}+ \ldots + x_n\alpha_n^{(j)} + \beta^{(j)} \quad &\forall 1 \leq j \leq r_1 \\
  &v_j=x_1\Re(\alpha_1^{(j)})+ \ldots + x_n\Re(\alpha_n^{(j)}) +\Re(\beta^{(j)}) \quad &\forall r_1+1 \leq j \leq r_1+r_2 \\
  &w_j=x_1\Im(\alpha_1^{(j)})+ \ldots + x_n\Im(\alpha_n^{(j)})+\Im(\beta^{(j)}) \quad &\forall r_1+1 \leq j \leq r_1+r_2
\end{align*}
we may write \eqref{LatticePointEquations} as
\begin{equation} \label{VolumeEquations}
\begin{aligned} 
&|u_j| \leq x, \quad |u_j+\zeta^{(j)}| \leq x \quad &\forall 1 \leq j \leq r_1 \\
&v_j^2+w_j^2 \leq x^2, \quad (v_j + \Re(\zeta^{(j)}))^2+(w_j + \Im(\zeta^{(j)}))^2 \leq x^2 \quad &\forall r_1+1 \leq j \leq r_1+r_2.
\end{aligned}
\end{equation}
As in the proof of~\cite{DiophantineEquations}*{Lemma 3.2} we use that
the Jacobian corresponding to the above change of variables is
$\frac{2^{r_2}}{\mathcal{N}(\mathfrak{a})\sqrt{|\Delta_K|}}$ where
$\Delta_K$ is the discriminant of $K$. We now need to count lattice
points in the area enclosed by the curves in
\eqref{VolumeEquations}. We observe that the inequalities concerning
the real conjugates describe the intersection of two lines of length
$2x$, and the inequalities concerning the complex conjugates describe
the intersection of two circles with radius $x$ and central distances
$|\zeta^{(j)}|$ for $r_1+1 \leq j \leq r_1 +r_2$.

We start with having a look at the lines described by $|u_j| \leq x$
and $|u_j+\zeta^{(j)}| \leq x$. We use that
$\house{\zeta} \leq \sqrt{x}$ and thus get a contribution between
$2x-\sqrt{x}$ and $2x$ to the volume. All in all we hence get a
contribution of
$(2x+\mathcal{O}(\sqrt{x}))^{r_1}=2^{r_1}x^{r_1}+\mathcal{O}(x^{r_1-\frac{1}{2}})$
for the lines. Now we come to the contribution of the intersecting
circles described by the inequalities $v_j^2+w_j^2 \leq x^2$ and
$(v_j + \Re(\zeta^{(j)}))^2+(w_j + \Im(\zeta^{(j)}))^2 \leq x^2$. An
obvious upper bound for the area enclosed by both of these circles is
the area of a full circle with radius $x$, i.e. $\pi x^2$. To get a
lower bound we again use that $\house{\zeta} \leq \sqrt{x}$ and we
compute the area enclosed by two circles with radius $x$ and central
distance $\sqrt{x}$. This area is given by
$$4\int_{\frac{\sqrt{x}}{2}}^{x}{\sqrt{x^2-t^2} \mathrm{d}t}=2\left[t\sqrt{x^2-t^2}+x^2\arctan\left(\frac{t}{\sqrt{x^2-t^2}}\right)\right]_{t=\frac{\sqrt{x}}{2}}^x=\pi x^2 + \mathcal{O}(x^{\frac{3}{2}}).$$
To get the error term we used that $\arctan(x)<x$ for any $x>0$ as
well as $\frac{\sqrt{x}}{2\sqrt{x^2-\frac{x}{4}}}<\frac{1}{\sqrt{x}}$
for $x>\frac{1}{2}$. Thus the contribution of all circles is
$(\pi x^2+\mathcal{O}(x^{\frac{3}{2}}))^{r_2} =
\pi^{r_2}x^{2r_2}+\mathcal{O}(x^{2r_2-\frac{1}{2}})$.
Altogether the curves in \eqref{VolumeEquations} therefore enclose an
area of
$$2^{r_1}\pi^{r_2}x^n+\mathcal{O}(x^{n-\frac{1}{2}}).$$
By enlarging and shrinking the area described by the curves in
\eqref{VolumeEquations} slightly (for details again see the proof
of~\cite{DiophantineEquations}*{Lemma 3.2}) we get
\begin{equation} \label{CountingLatticPointsFinalEquation}
\PP(\mathfrak{a},\zeta,x)=\frac{\pi^{r_2}2^{r_1+r_2}}{\mathcal{N}(\mathfrak{a})\sqrt{|\Delta_K|}}x^n+\mathcal{O}\left(\frac{x^{n-\frac{1}{2}}}{\mathcal{N}(\mathfrak{a})}+\frac{x^{n-1}}{\mathcal{N}(\mathfrak{a})^{1-\frac1n}}\right).
\end{equation}
\textbf{2. Sifting by Prime Ideals} \\ \hfill For any details in this
part of the proof we refer the reader to the work of Rieger,
especially the proof of~\cite{VerallgemeinerungDer3}*{Satz 16}. Our aim
here is just to point out explicit values for the constants appearing
in Rieger's proof.

In his proof Rieger chose the parameter $z \leq x^{\frac{1}{3}}$. This
choice will not work with our error term
in~\eqref{CountingLatticPointsFinalEquation} and we need to take the
slightly worse bound $z \leq x^{\frac{1}{4}-\epsilon}$, where $\epsilon$ is an arbitrarily small positive number. Applying~\cite{VerallgemeinerungDer1}*{Satz 2} we get an error term of
$\mathcal{O}\left(\frac{x^{n-2\epsilon}}{\mathcal{N}(\mathfrak{a})}\right)$
in equation \eqref{CountingLatticPointsFinalEquation} and it remains
to work out an upper bound for the main term.

The main term is of the form $\frac{\pi^{r_2}2^{r_1+r_2}}{\sqrt{|\Delta_K|}}\frac{x^n}{Z}$ where as in~\cite{VerallgemeinerungDer3}*{p. 86  equation (119)} $Z$ is bounded from below by
\begin{equation} \label{ZBound}
Z \geq \sum_{\mathcal{N}(\mathfrak{a})\leq \sqrt{z}}\frac{1}{\mathcal{N}(\mathfrak{a})} \sum_{\substack{\mathcal{N}(\mathfrak{b})\leq \sqrt{z} \\ (\mathfrak{b},(\zeta))=\mathcal{O}_K}}\frac{1}{\mathcal{N}(\mathfrak{b})}.
\end{equation} 
Finding lower bounds for sums of the above type works as
in~\cite{VerallgemeinerungDer2}*{p. 160 equation (40)} where we
additionally use that
\begin{align*}
  \prod_{\mathcal{N}(\mathfrak{p})\leq \sqrt{z}}\left(1-\frac{1}{\mathcal{N}(\mathfrak{p})}\right)&=\prod_{\mathcal{N}(\mathfrak{p})\leq \sqrt{z}}\left(1-\frac{1}{\mathcal{N}(\mathfrak{p})^2}\right)\prod_{\mathcal{N}(\mathfrak{p})\leq \sqrt{z}}\left(1+\frac{1}{\mathcal{N}(\mathfrak{p})}\right)^{-1} \\
                                                                                                  &\geq \left(\frac{6}{\pi^2}\right)^n \prod_{\mathcal{N}(\mathfrak{p})\leq \sqrt{z}}\left(1+\frac{1}{\mathcal{N}(\mathfrak{p})}\right)^{-1}.
\end{align*}
The constant $\left(\frac{6}{\pi^2}\right)^n$ arises from the
inequality
$$\prod_{\mathcal{N}(\mathfrak{p})\leq \sqrt{z}}\left(1-\frac{1}{\mathcal{N}(\mathfrak{p})^2}\right) \geq \prod_{p \in \mathbb{P}}\left(1-\frac{1}{p^2}\right)^n =\left(\frac{6}{\pi^2}\right)^n.$$
For the last inequality we used that in the ring of integers of a
number field of degree $n$ there are at most $n$ prime ideals whose
norm is a power of $p$ for fixed $p \in \mathbb{P}$. This follows
basically from the fact that any prime ideal in $\mathcal{O}_K$ lies
over a prime ideal in $\mathbb{Z}$ and there are at most $n$ prime
ideals in $\mathcal{O}_K$ with this property (see for
example~\cite{ElementaryAnd}*{Proposition 4.2 and Corollary 2 on
p. 148}). We note that in the case of the Gaussian
integers this bound can be improved. Since we will use it later we
also give the improved bound for $K=\mathbb{Q}(i)$ here:
\begin{equation} \label{ImprovedZIBoundTatuzawa} \begin{split}
\prod_{\mathcal{N}(\mathfrak{p})\leq \sqrt{z}}\left(1-\frac{1}{\mathcal{N}(\mathfrak{p})^2}\right)&=\prod_{\substack{p \leq \sqrt{z} \\ p \equiv 1 \bmod 4}}\left(1-\frac{1}{p^2}\right)^2\prod_{\substack{p \leq \sqrt[4]{z} \\ p \equiv 3 \bmod 4}}\left(1-\frac{1}{p^4}\right)\\
&\geq \prod_{p \equiv 1 \bmod 4}\left(1-\frac{1}{p^2}\right)^2\prod_{p \equiv 3 \bmod 4}\left(1-\frac{1}{p^4}\right) \\
&= L(\chi_1,2)^{-1}L(\chi_2,2)^{-1} \geq 0.88493,
\end{split}
\end{equation}
where $\chi_1$ and $\chi_2$ are the two Dirichlet Characters $\bmod$
$4$. Another ingredient we will use is Mertens' formula for number
fields (see~\cite{GeneralisedMertens}*{Theorem 5}) in the form
\begin{equation} \label{MertensNF}
\prod_{\mathcal{N}(\mathfrak{p})\leq x}\left(1-\frac{1}{\mathcal{N}(\mathfrak{p})}\right)^{-1} = e^{\gamma+ \log \rho_K} \log x + \mathcal{O}\left(\frac{1}{\log^2x}\right),
\end{equation}
where $\rho_K$ is the residue of the Dedekind zeta function of $K$ at
$s=1$. In his proof of a lower bound for $Z$
in~\cite{VerallgemeinerungDer2}*{p. 160 equation (40)} Rieger uses
that there exists a fixed constant $c$ independent of the ideal
$\mathfrak{f}$ such that
\begin{equation} \label{PracharInequality}
\sum_{\substack{\mathcal{N}(\mathfrak{m})\leq x \\ (\mathfrak{m},\mathfrak{f})=\mathcal{O}_K}}\frac{1}{\mathcal{N}(\mathfrak{m})} \geq c\prod_{\substack{\mathcal{N}(\mathfrak{p})\leq x \\ \mathfrak{p} \nmid \mathfrak{f}}}\left(1-\frac{1}{\mathcal{N}(\mathfrak{p})}\right)^{-1}
\end{equation}
(see~\cite{VerallgemeinerungDer2}*{Hilfssatz 5}). For a proof, Rieger
refers to the analogous result in $\mathbb{Z}$ proved in \cite{Primzahlverteilung}*{II
Lemma 4.1} by using Mertens' inequality
\eqref{MertensNF} and that there exists a constant $\alpha$ such that
\begin{equation} \label{HarmonicSeries}
\sum_{\mathcal{N}(\mathfrak{m})\leq x}\frac{1}{\mathcal{N}(\mathfrak{m})} = \alpha\log x +\mathcal{O}(1)
\end{equation}
(see~\cite{VerallgemeinerungDer2}*{p. 158 equation (29)}). Since we
need the constant $c$ in inequality \eqref{PracharInequality} explicitly
we give Prachar's proof of~\cite{Primzahlverteilung}*{II Lemma
4.1} for this case.  To begin with we need to
determine the constant $\alpha$ in equation \eqref{HarmonicSeries}. Applying the
Tauberian Theorem of Delange-Ikehara (see \cite{ElementaryAnd}*{p. 464 Theorem I})
and partial summation we see that $\alpha=\rho_K$. We also note that in the case of the Gaussian
integers the left hand side of \eqref{HarmonicSeries} takes the form
\[\frac{1}{4}\sum_{m\leq x}\frac{r_2(m)}{m}\]
where $r_2(m)$ counts the number of representations of $m$ as the sum
of two squares. A well known result of
Sierpi\'{n}ski~\cite{UeberDieSummierung} states that
\[\sum_{m\leq x}\frac{r_2(m)}{m}=\pi\log x +K +\mathcal{O}\left(\frac{1}{\sqrt{x}}\right)\]
where $K>0$ is the Sierpi\'{n}ski constant.

Combining equations \eqref{HarmonicSeries} and \eqref{MertensNF} 
we have that for $\mathfrak{f}=\mathcal{O}_K$ inequality
\eqref{PracharInequality} is satisfied with $c=e^{-\gamma} +o(1)$. 
Following Prachar's proof we will show that we
can keep this constant also in the case
$\mathfrak{f} \neq \mathcal{O}_K$. Take a prime divisor
$\mathfrak{q}|\mathfrak{f}$ with $\mathcal{N}(\mathfrak{q}) \leq x$,
then for sufficiently large x we have
\begin{align*}
\prod_{\substack{\mathcal{N}(\mathfrak{p})\leq x \\ \mathfrak{p} \neq \mathfrak{q}}}\left(1-\frac{1}{\mathcal{N}(\mathfrak{p})}\right)^{-1}&=\prod_{\mathcal{N}(\mathfrak{p})\leq x}\left(1-\frac{1}{\mathcal{N}(\mathfrak{p})}\right)^{-1}\left(1-\frac{1}{\mathcal{N}(\mathfrak{q})}\right) \\ 
&\leq (e^{\gamma}+o(1))\left(1-\frac{1}{\mathcal{N}(\mathfrak{q})}\right)\sum_{\mathcal{N}(\mathfrak{m}) \leq x}\frac{1}{\mathcal{N}(\mathfrak{m})}.
\end{align*}
Since
\[\sum_{\substack{\mathcal{N}(\mathfrak{m}) \leq x \\ \mathfrak{q} \nmid \mathfrak{m}}}\frac{1}{\mathcal{N}(\mathfrak{m})} = \sum_{\mathcal{N}(\mathfrak{m}) \leq x}\frac{1}{\mathcal{N}(\mathfrak{m})}-\sum_{\substack{\mathcal{N}(\mathfrak{m}) \leq x \\ \mathfrak{q} | \mathfrak{m}}}\frac{1}{\mathcal{N}(\mathfrak{m})} \geq \sum_{\mathcal{N}(\mathfrak{m}) \leq x}\frac{1}{\mathcal{N}(\mathfrak{m})}-\frac{1}{\mathcal{N}(\mathfrak{q})}\sum_{\mathcal{N}(\mathfrak{m}) \leq x}\frac{1}{\mathcal{N}(\mathfrak{m})}\]
we conclude that
$$\prod_{\substack{\mathcal{N}(\mathfrak{p})\leq x \\ \mathfrak{p} \neq \mathfrak{q}}}\left(1-\frac{1}{\mathcal{N}(\mathfrak{p})}\right)^{-1} \leq (e^{\gamma}+o(1))\sum_{\substack{\mathcal{N}(\mathfrak{m}) \leq x \\ \mathfrak{q} \nmid \mathfrak{m}}}\frac{1}{\mathcal{N}(\mathfrak{m})}.$$
Iterating for all of the finitely many prime factors
$\mathfrak{q}|\mathfrak{f}$ with $\mathcal{N}(\mathfrak{q})\leq x$
gives the desired result. 

Then we have
\begin{align*}
\sum_{\substack{\mathcal{N}(\mathfrak{m})\leq \sqrt{z} \\ (\mathfrak{m},\mathfrak{f})=\mathcal{O}_K}}\frac{1}{\mathcal{N}(\mathfrak{m})} &\geq
(e^{-\gamma}+o(1))\prod_{\substack{\mathcal{N}(\mathfrak{m})\leq \sqrt{z} \\ \mathfrak{p}\nmid \mathfrak{f}}}\left(1-\frac{1}{\mathcal{N}(\mathfrak{p})}\right)^{-1}\\
&=(e^{-\gamma}+o(1))\prod_{\mathcal{N}(\mathfrak{p})\leq \sqrt{z}}\left(1-\frac{1}{\mathcal{N}(\mathfrak{p})}\right)^{-1}\prod_{\substack{\mathcal{N}(\mathfrak{p})\leq \sqrt{z} \\ \mathfrak{p}|\mathfrak{f}}}\left(1-\frac{1}{\mathcal{N}(\mathfrak{p})}\right) \\
&\geq \left(\left(\frac{6}{\pi^2}\right)^n\frac{\rho_K}{2}+o(1)\right)\prod_{\mathfrak{p}|\mathfrak{f}}\left(1+\frac{1}{\mathcal{N}(\mathfrak{p})}\right)^{-1}\log z\left(1+\mathcal{O}\left(\frac{1}{\log ^2z}\right)\right)
\end{align*} 
Using this and \eqref{HarmonicSeries} in inequality \eqref{ZBound} we get
\begin{equation*}
Z \geq \log^2(z)\left(\left(\frac{6}{\pi^2}\right)^{n}\frac{\rho_K^2}{4}+o(1)\right)\prod_{\mathfrak{p}| (\zeta)}\left(1+\frac{1}{\mathcal{N}(\mathfrak{p})}\right)^{-1}\left(1+\mathcal{O}\left(\frac{1}{\log^2 z}\right)\right) 
\end{equation*}
and we note that in the case $K=\mathbb{Q}(i)$ we may replace $\left(\frac{6}{\pi^2}\right)^n$ with the bound from \eqref{ImprovedZIBoundTatuzawa}. Altogether we thus get
$$\PP(\zeta,x)\leq \kappa \frac{x^n}{\log^2x}\prod_{\mathfrak{p}| (\zeta)}\left(1+\frac{1}{\mathcal{N}(\mathfrak{p})}\right)(1+o(1)),$$
where in the general case we may choose
$$\kappa=\frac{2^{r_1+r_2+2}\pi^{r_2+2n}}{\left(\frac{1}{4}-\epsilon\right)^2\sqrt{|\Delta_K|}6^n\rho_K^2}.$$
In the case $K=\mathbb{Q}(i)$ we can replace the factor $\frac{\pi^{2n}}{6^n}$ by $0.88493^{-1}$ and with $r_1=0$, $r_2=1$, $\rho_{\mathbb{Q}(1)}=\frac{\pi}{4}$, $\Delta_{\mathbb{Q}(i)}=-4$ and $\epsilon = 10^{-10}$ we see that $\kappa=368.333$ is admissible.
\end{proof}
Before we start proving Theorem \ref{MainTheorem} we note that if $\zeta\in\mathcal{O}_K$ is non-zero and not a unit
then $\house{\zeta} >1$. Furthermore the following lemma gives a connection between the norm and the house function: 

\begin{lem} \label{NormNormLemma}
For a fixed non-unit $\xi \in \mathcal{O}_K$ and $k \in \mathbb{N}$ 
$$\mathcal{N}(\xi^k-1) \leq c\house{\xi}^{kn}$$
for a constant $c$ depending only on $K$.
\end{lem}
\begin{proof}
We have 
$$\mathcal{N}(\xi^k-1) \leq \house{\xi^k-1}^n \leq \left(\house{\xi^k}+1\right)^n.$$
The very crude bound $\binom{n}{k} \leq 2^n$ gives the lemma with $c=(n+1)2^n$.
\end{proof}

Now we have all the tools we need to bound the number of elements in the sets $A$ and
$B$.  First we provide a lower bound for the number of elements of the
set $A$ in \eqref{SetA}. We define
$L_{\xi}:=\frac{\log \frac{\sqrt{x}}{2}}{\log \house{\xi}}$. Mitsui~\cite{GeneralizedPrime}*{Main Theorem p. 35} proved that the
number of prime elements $\pi \in \mathcal{P}_K$ with
$\house{\pi} \leq x$ is asymptotically of size
$$\frac{\omega x^n}{nh_k2^{r_1}R\log x}=:c_1\frac{x^n}{\log x}.$$
Here $\omega$ is the number of roots of unity in $K$, $h_K$ is the class number and $R$ the regulator of $K$. We thus have
\begin{equation} \label{RXnBoundLower}
\sum_{\house{\zeta} \leq x}{r_x(\zeta,\xi)} \geq \sum_{k \leq L_{\xi}}\sum_{\substack{\pi \in \mathcal{P}_K \\ \house{\pi} \leq x-\house{\xi}^k}}{1} \sim c_1 \sum_{k \leq L_{\xi}}{\frac{(x-\house{\xi}^k)^n}{\log (x-\house{\xi}^k)}} \geq \frac{c_1}{\log x}\sum_{k \leq L_{\xi}}{(x-\house{\xi}^k)^n} \sim \frac{c_1x^n}{2\log \house{\xi}} 
\end{equation}
where we used that 
$$\sum_{k \leq L_{\xi}}{(x-\house{\xi}^k)^n} =\frac{x^n \log x}{2\log \house{\xi}}\left(1+\mathcal{O}\left(\frac{1}{\log x}\right)\right).$$
On the other hand to get an upper bound we consider
\begin{equation} \label{RXnBoundUpper}
\sum_{\house{\zeta} \leq x}{r_x(\zeta,\xi)} \leq \sum_{k \leq L_{\xi}}\sum_{\house{\pi} \leq x}{1} \sim c_1\sum_{k \leq L_{\xi}}{\frac{ x^n}{\log x}} \sim \frac{c_1 x^n}{2\log \house{\xi}}.
\end{equation}

Now we count the elements in the set $B$ defined in
\eqref{SetB}. Following Romanov's approach we use Theorem
\ref{NFSieve} to count pairs of primes which sum to
$\xi^{k_1}-\xi^{k_2}$ for fixed $k_1, k_2 \in \mathbb{N}$.
We thus have
\begin{equation} \label{RGx2Bound} \#B \leq \#A + \kappa\frac{x^n}{\log^2
    x}\prod_{\xi \in
    \mathfrak{p}}{\left(1+\frac{1}{\mathcal{N}(\mathfrak{p})}\right)}\sum_{k_1
    < k_2 \leq L_{\xi}}\prod_{\xi^{k_2-k_1}-1 \in
    \mathfrak{p}}{\left(1+\frac{1}{\mathcal{N}(\mathfrak{p})}\right)}
\end{equation}
and we need to take care of the sum over $k_1$ and $k_2$. Observing that for the innermost product we have
$$\prod_{\xi^{k_2-k_1}-1 \in \mathfrak{p}}{\left(1+\frac{1}{\mathcal{N}(\mathfrak{p})}\right)}=\sum_{\substack{\xi^{k_2-k_1}-1 \in \mathfrak{a} \\ \mu^2(\mathfrak{a})=1}}{\frac{1}{\mathcal{N}(\mathfrak{a})}}$$ 
where $\mu$ as usual denotes the M\"obius function. Interchanging summation we end up looking for a bound for
\begin{equation} \label{EstablishConvergentSumEquation}
\sum_{\mu^2(\mathfrak{a})=1}{\frac{1}{\mathcal{N}(\mathfrak{a})}}\sum_{\substack{k_1<k_2 \leq L_{\xi} \\ \ord_{\xi}(\mathfrak{a})|(k_2-k_1)}}{1} \leq \frac{\log^2 x}{8\log^2 \house{\xi}}(1+o(1))\sum_{\substack{\mu^2(\mathfrak{a})=1 \\ (\mathfrak{a},(\xi))=\mathcal{O}_K}}{\frac{1}{\mathcal{N}(\mathfrak{a})\ord_{\xi}(\mathfrak{a})}},
\end{equation}
where $\ord_{\xi}(\mathfrak{a}):=\min \{k \geq 1: \xi^k-1 \in \mathfrak{a}\}$. General versions of the sum appearing on the right hand side of the last inequality were treated by Ram Murty et.al.~\cite{VariationsOn}. Since we later want to work out the constants we will give a detailed proof of the convergence of the above sum by following the lines of the proof of~\cite{AdditiveNumber}*{Lemma 7.8} for the case $K=\mathbb{Q}$.

\begin{lem} \label{SumConvergenceLemma}
For any non-unit $\xi \in \mathcal{O}_K$ the series
$$\sum_{e=1}^{\infty}\frac1e \sum_{\substack{\mu^2(\mathfrak{a})=1 \\ (\mathfrak{a},(\xi))=\mathcal{O}_K \\ \ord_{\xi}(\mathfrak{a})=e}}{\frac{1}{\mathcal{N}(\mathfrak{a})}}$$ 
converges. An upper bound for the sum is given by $2^nn!e^{n\gamma}(1+o(1))$.
\end{lem}

\begin{proof}
Let $D$ be a product of principal ideals of the form $(\xi^k-1)$ i.e.
\begin{equation} \label{PEquation}
D:=\prod_{k \leq x}{(\xi^k-1)}.
\end{equation}
We consider the sum 
\begin{equation} \label{EDefinitionEquation}
E(x)=\sum_{k \leq x}\sum_{\substack{\mu^2(\mathfrak{a})=1
    \\(\mathfrak{a},(\xi))=\mathcal{O}_K \\ \ord_{\xi}(\mathfrak{a})
    =k }}{\frac{1}{\mathcal{N}(\mathfrak{a})}} \leq
\prod_{\mathfrak{p}\mid D}{\left(1+\frac{1}{\mathcal{N}(\mathfrak{p})}\right)},
\end{equation}
where the last sum runs over all prime ideals $\mathfrak{p}$ dividing
$D$.  Using Lemma \ref{NormNormLemma} we get for the number
$\omega(D)$ of distinct prime divisors of $D$
$$2^{\omega(D)} \leq \mathcal{N}(D) = \prod_{k \leq x}{\mathcal{N}(\xi^k-1)} \leq \prod_{k \leq x}{c\house{\xi}^{kn}} = c^x\prod_{k \leq x}{c_{\xi}^k} \leq c^x c_{\xi}^{x^2}$$
where $c_{\xi}= \house{\xi}^n$ and with $c$ as in Lemma \ref{NormNormLemma}. Hence $\omega(D) \leq x^2\frac{\log c_{\xi}}{\log 2}(1+o(1))$. Using that there are at most $n$ prime ideals in $K$ whose norm is a power of $p \in \mathbb{Z}$ we arrive at
\begin{equation} \label{EEquation} \begin{split}
E(x) &\leq \prod_{i=1}^{\frac{\omega(D)}{n}}{\left(1+\frac{1}{p_i}\right)^n} \leq \prod_{i=1}^{\frac{\log c_{\xi}(1+o(1))}{n\log 2}x^2}{\left(1+\frac{1}{p_i}\right)^n}\\
&= \prod_{i=1}^{\frac{\log c_{\xi}(1+o(1))}{n\log 2}x^2}{\left(1-\frac{1}{p_i^2}\right)^n}\prod_{i=1}^{\frac{\log c_{\xi}(1+o(1))}{n\log 2}x^2}{\left(1-\frac{1}{p_i}\right)^{-n}} \leq 2^ne^{n\gamma}\log^nx(1+o(1)) 
\end{split}
\end{equation}
where $p_i$ is the $i$-th prime and we used Mertens' formula as well as the fact that the first product converges. As in the proof of~\cite{AdditiveNumber}*{Lemma 7.8} we use partial summation to derive the final result:

\begin{align*}
\sum_{e=1}^{\infty} \frac{1}{e} \sum_{\substack{\mu^2(\mathfrak{a})=1 \\(\mathfrak{a},(\xi))=\mathcal{O}_K \\\ord_{\xi}(\mathfrak{a}) =e }}{\frac{1}{\mathcal{N}(\mathfrak{a})}}&=\left[\frac{E(x)}{x}\right]_{x=1}^{\infty}+\int_1^{\infty}{\frac{E(t)}{t^2}\text{d}t} \leq 2^ne^{n\gamma}(1+o(1))\int_1^{\infty}{\frac{(\log t)^n}{t^2}\text{d}t} \\
&= 2^nn!e^{n\gamma}(1+o(1)).
\end{align*}
\end{proof}

\begin{proof}[Proof of Theorem \ref{MainTheorem}]
By \eqref{RXnBoundLower} we have that $\sum_{\house{\zeta} \leq x}{r_x(\zeta)} \gg x^n$ and \eqref{RGx2Bound} and \eqref{EstablishConvergentSumEquation} together with Lemma \ref{SumConvergenceLemma} prove that $\sum_{\house{\zeta} \leq x}{r_x(\zeta)}^2 \ll x^n$. Plugging this in \eqref{CSInequality} shows that 
$$\sum_{\house{\zeta} \leq x}{\eta_x(\zeta)} \gg x^n.$$ 
To finish the proof of Theorem \ref{MainTheorem} we make use of the fact that there are $\mathcal{O}(x^n)$ algebraic integers $\zeta \in \mathcal{O}_K$ such that $\house{\zeta} \leq x$ (for a proof in the case of counting totally positive algebraic integers see~\cite{DiophantineEquations}*{Lemma 3.2}, the proof of the general case works along the same lines as is mentioned after the proof of~\cite{DiophantineEquations}*{Lemma 3.2}). We may also see this by taking $\mathfrak{a}=\mathcal{O}_K$, $\beta \in \mathcal{O}_K$ arbitrary and $\zeta = 0$ in the first part of Theorem \ref{NFSieve} from where we get that the implied constant my be chosen as $\frac{\pi^{r_2}2^{r_1+r_2}}{\sqrt{|\Delta_K|}}$. Using this and the explicit constants in \eqref{RXnBoundLower}, \eqref{RGx2Bound} and Lemma \ref{SumConvergenceLemma} we get a lower bound for the lower density of the form
$$\frac{\left(\frac{\omega}{2nh_K2^{r_1}R\log \house{\xi}}\right)^2}{\frac{\pi^{r_2}2^{r_1+r_2}}{\sqrt{|\Delta_K|}}\left(\frac{\omega}{2nh_K2^{r_1}R\log \house{\xi}}+\frac{2^{r_1+r_2+2}\pi^{r_2+2n}}{\left(\frac{1}{4}-\epsilon\right)^2\sqrt{|\Delta_K|}6^n\rho_K^2}\prod_{\xi \in \mathfrak{p}}\left(1+\frac{1}{\mathcal{N}(\mathfrak{p})}\right)\frac{1}{8\log^2\house{\xi}}2^ne^{n\gamma}n!\right)}.$$ 
\end{proof}

\section{The special case of the Gaussian integers}

To get a better constant for Gaussian integers of the form $\pi+(1+i)^k$ we need to be more careful in estimating Romanov's sum than we have been in the proof of Lemma \ref{SumConvergenceLemma}. We will apply the following two results. 

\begin{lem} \label{FirstPartLemma}
The following bound holds:
$$\sum_{e \leq 200}\frac{1}{e}\sum_{\substack{\mu^2(\mathfrak{a})=1 \\ (\mathfrak{a},(1+i))=\mathbb{Z}[i] \\ \ord_{1+i}(\mathfrak{a})=e }}\frac{1}{\mathcal{N}(\mathfrak{a})} \leq 1.27096.$$
\end{lem}

The proof of Lemma \ref{FirstPartLemma} is done by explicit calculation using a computer algebra system. The second result is a bound for the remainder term of Romanov's sum and we will derive it by applying the ideas used by Chen and Sun for sums of primes and powers of $2$ (see~\cite{OnRomanoffs}*{Lemma 4}).

\begin{lem} \label{ChenSunLemma}
$$\sum_{e>200}\frac{1}{e}\sum_{\substack{\mu^2(\mathfrak{a})=1 \\ (\mathfrak{a},(1+i))=\mathbb{Z}[i] \\ \ord_{1+i}(\mathfrak{a})=e}}\frac{1}{\mathcal{N}(\mathfrak{a})} \leq 0.31038.$$
\end{lem}

\begin{proof}
As in Lemma \ref{SumConvergenceLemma} we start with a product of the form
$$D=\prod_{k=\lceil \frac{x}{2}\rceil}^x((1+i)^k-1).$$
Here as Chen and Sun \cite{OnRomanoffs} we note that it suffices to
start with $k=\lceil \frac{x}{2} \rceil$. We want to deduce an upper
bound for the largest prime factor dividing $\mathcal{N}(D)$. Suppose that $D$ has exactly $m$ distinct prime
factors $\pi_1, \ldots, \pi_m$ where we note that their norm is odd
since $1+i$ and its associates do not divide $D$. Then we have
$$p_1^2 \cdot \ldots \cdot p_{\lfloor\frac{m}{2}\rfloor}^2 \leq \mathcal{N}(\pi_1) \cdot \ldots \cdot \mathcal{N}(\pi_m) \leq \prod_{k=\lceil\frac{x}{2}\rceil}^x{(\house{1+i}^k+1)^2} \leq 2\prod_{k=\lceil \frac{x}{2}\rceil}^x{2^k},$$
where the $p_i$ are the
first $\lfloor \frac{m}{2} \rfloor$ odd primes and the last inequality holds for $x \geq 6$. Taking the logarithm
on both sides of the last inequality yields
$$2\sum_{i=1}^{\lfloor \frac{m}{2} \rfloor}\log p_i \leq \log2\left(1+\sum_{k=\lceil \frac{x}{2} \rceil}^xk\right).$$
As Chen and Sun \cite{OnRomanoffs} we will first use computational
methods to get a lower bound for the prime number $p_{\lfloor \frac{m}{2} \rfloor}$ and
apply~\cite{SharperBounds}*{Corollary 2*} to get an explicit upper
bound for $p_{\lfloor \frac{m}{2} \rfloor}$ altogether. Using a computer algebra system it is easy
to verify that for $x=200$ the product $D$ has exactly $419$ different
prime factors and that the $209$-th odd prime is $p_{209}=1291$. This
together with the bound from~\cite{SharperBounds}*{Corollary 2*}
implies that
$$2p_{\lfloor \frac{m}{2} \rfloor}\left(1-\frac{2}{5\cdot \log(1291)}\right)<\log 2\left(1+\sum_{\lceil \frac{x}{2}\rceil}^xk\right)$$
and hence for $x \geq 200$
$$p_{\lfloor \frac{m}{2} \rfloor}<\frac{\log 2}{2}\left(1-\frac{2}{5\cdot \log(1291)}\right)^{-1}\frac{15151}{40000}x^2<0.13904x^2.$$

For $E(x)$ defined as in \eqref{EDefinitionEquation} in the proof of Lemma \ref{SumConvergenceLemma} analogously as in equation \eqref{EEquation} we get that
$$E(x) \leq \prod_{i=1}^{\lfloor \frac{m}{2} \rfloor}\left(1-\frac{1}{p_i^2}\right)^2\prod_{i=1}^{\lfloor \frac{m}{2} \rfloor}\left(1-\frac{1}{p_i}\right)^{-2}.$$
According to~\cite{OnRomanoffs}*{Lemma 3} we have for $x\geq 74$
$$\prod_{3 \leq p \leq x}\left(1-\frac{1}{p}\right)^{-1} \leq 0.922913686 \log x.$$
Furthermore with $\lfloor \frac{m}{2}\rfloor =209$ we have that
$$\prod_{i=1}^{\lfloor \frac{m}{2} \rfloor}\left(1-\frac{1}{p_i^2}\right)^2 \leq 0.65715$$
by direct computation. Altogether this implies that
$$E(x) \leq 2.23897 \log^2x- 4.41745 \log x + 2.17891$$
Using a computer algebra system it is easy to verify that $E(200) \geq 3.33018$ and using partial summation as in Lemma \ref{SumConvergenceLemma} we get that
\begin{align*}
\sum_{e>200}\frac{1}{e}\sum_{\substack{\mu^2(\mathfrak{a})=1 \\ (\mathfrak{a},(1+i))=\mathbb{Z}[i] \\ \ord_{1+i}(\mathfrak{a})=e}}\frac{1}{\mathcal{N}(\mathfrak{a})} &\leq -\frac{E(200)}{200}+\int_{200}^{\infty}{\frac{2.23897 \log^2x- 4.41745 \log x + 2.17891}{x^2} \text{d}x} \\
&\leq 0.31038.
\end{align*}
\end{proof}

\begin{proof}[Proof of Theorem \ref{ZIExplicitBoundTheorem}]
We need to work out constants $c_1$, $c_2$ and $c_3$ with 
$$\sum_{\house{\zeta} \leq x}{r_x(\zeta,1+i)} \sim c_1 x^2 \text{, }\sum_{\house{\zeta} \leq x}{r_x(\zeta,1+i)^2} \leq c_2 x^2 \text{ and } \sum_{\house{\zeta} \leq x}1 \sim c_3x^n.$$
Using inequality \eqref{CSInequality} a lower bound for the lower density is then given by $\frac{c_1^2}{c_2c_3}$. 

Computing the constant $c_3$ reduces to counting integral lattice
points within the circle of radius $x$. It is well known that
$\sum_{\house{\zeta} \leq x}1 \sim \pi x^n$ where we note that the
asymptotic rate of growth of the error term is linked to the Gauss
circle problem which is frequently listed among famous open problems
in number theory (see for example~\cite{UnsolvedProblems}*{p. 365
ff}). The constant $c_3$ may therefore be chosen as
$c_3=\pi$.

With $\house{1+i}=\sqrt{2}$ equation \eqref{RXnBoundLower} yields that
$c_1=\frac{2}{\log 2}$ (note that there are four roots of unity in
$\mathbb{Z}[i]$, the regulator as well as the class number are $1$,
$r_1=0$ and $r_2=1$) and it remains to determine the constant $c_2$.
 
The constant $c_2$ is the sum of the constant in equation
\eqref{RXnBoundUpper} which is $\frac{2}{\log 2}$ and the product of
the constant $\tilde{c}_1=\frac{1}{2\log^22}$ in equation
\eqref{EstablishConvergentSumEquation}, the constant
$\kappa=368.333$ from Theorem \ref{NFSieve}, the sum of the constants in Lemmas \ref{FirstPartLemma} and \ref{ChenSunLemma} and the constant
\[\tilde{c_4}=\prod_{1+i \in
  \mathfrak{p}}\left(1+\frac{1}{\mathcal{N}(\mathfrak{p})}\right)=\frac{3}{2}.\]
\end{proof}

It remains to prove Theorem \ref{ErdosVanDerCorputAnalogueTheorem}. To
do so we will apply Erd\H{o}s' idea (see the proof of~\cite{OnIntegersRelated}*{Theorem
3}) of using covering congruences, i.e. a system of
$l$ congruences $a_j \bmod m_j$ for $a_j,m_j \in \mathbb{N}$ such that
$$\mathbb{N} \subset \bigcup_{j=1}^l\bigcup_{\lambda = 0}^{\infty}(a_j+\lambda m_j).$$
In Nathanson's book the details of Erd\H{o}s' proof were worked out
(see~\cite{AdditiveNumber}*{p. 204 ff}) and the reader can find a proof of the
following Lemma there (see~\cite{AdditiveNumber}*{Lemma 7.11}).

\begin{customlem}A \label{CoveringCongruencesLemma}
The six congruences $0 \bmod 2$, $0 \bmod 3$, $1 \bmod 4$, $3 \bmod 8$, $7 \bmod 12$ and $23 \bmod 24$ form a set of covering congruences.
\end{customlem}

As Erd\H{o}s we will use the covering congruences in Lemma \ref{CoveringCongruencesLemma} to construct a set with positive density in $\mathbb{Z}[i]$ consisting of Gaussian integers of odd norm not of the form $\pi + (1+i)^k$ for $\pi\in \mathcal{P}_{\mathbb{Q}(i)}$.

\begin{proof}[Proof of Theorem \ref{ErdosVanDerCorputAnalogueTheorem}]
We start by choosing prime elements $\pi_j \in \mathbb{Z}[i]$ such that $(1+i)^{m_j}-1 \in (\pi_j)$ where the $m_j$ are the moduli appearing in the covering congruences in Lemma \ref{CoveringCongruencesLemma}. We will use the following choice of primes $\pi_j$:
\begin{align*}
(1+i)^2 &\equiv 1 \bmod (2+i) &(1+i)^3 &\equiv 1 \bmod (2+3i) \\ 
(1+i)^4 &\equiv 1 \bmod (1+2i) &(1+i)^8 &\equiv 1 \bmod (3) \\
(1+i)^{12} &\equiv 1 \bmod (3+2i) &(1+i)^{24} &\equiv 1 \bmod (7).
\end{align*}

By the Chinese Remainder theorem the following congruences describe a unique residue class modulo $m = -(1+i)\prod_{j=1}^6{\pi_j}=990+990i$ in $\mathbb{Z}[i]$:
\begin{align*}
x &\equiv 1 \bmod (2+i) &x &\equiv 1 \bmod (2+3i) &x &\equiv 1+i \bmod (1+2i) \\
x &\equiv -2+2i \bmod (3) &x &\equiv 8-8i \bmod (3+2i) &x &\equiv 2048-2048i \bmod (7) 
\end{align*}
and $x \equiv 1 \bmod (1+i)$. Note that the first $6$ of the previous $7$ congruences are of the form $(1+i)^{a_j} \bmod \pi_j$ where the $a_j$ are the residue classes from Lemma \ref{CoveringCongruencesLemma}. Since the congruences in Lemma \ref{CoveringCongruencesLemma} are a covering system and by our choice of $x$ whenever we subtract a power of $1+i$ from any element in our arithmetic progression we get an element divisible by one of the primes $\pi_j$.

Because of $x \equiv 1 \bmod (1+i)$ we have that $\mathcal{N}(x)$ is odd. This holds true since $x$ is of the form $1+(1+i)(a+ib)$ for $a,b \in \mathbb{Z}$, $(1+i)(a+ib)=a-b +i(a+b)$ and $a-b$ and $a+b$ are of the same parity.

Thus non of the elements of the residue class $x \bmod (m)$ is of the form $\pi +(1+i)^k$ for $\pi \in \mathcal{P}_{\mathbb{Q}(i)}$ with exceptions being algebraic integers of the form $\pi_j+(1+i)^k$ for $1 \leq j \leq 6$. Since those $6$ sets have density $0$ in $\mathbb{Z}[i]$ we are done.
\end{proof}

\section*{Acknowledgment}

The second author is supported by the Austrian Science Fund (FWF): W1230, Doctoral Program `Discrete Mathematics'.

Parts of this research work were done when the second author was
visiting the Institut \'Elie Cartan de Lorraine of the University of
Lorraine. The author thanks the institution for its hospitality.


\end{document}